\documentclass[11pt,oneside,english]{amsart}
\usepackage[T1]{fontenc}
\usepackage[utf8]{inputenc}
\usepackage{color}
\usepackage{babel}
\usepackage{amstext}
\usepackage{amsthm}
\usepackage{amssymb}
\usepackage{graphicx}
\usepackage[caption=false]{subfig}
\usepackage[unicode=true,pdfusetitle,
 bookmarks=true,bookmarksnumbered=false,bookmarksopen=false,
 breaklinks=false,pdfborder={0 0 0},backref=false,colorlinks=true]
 {hyperref}
\hypersetup{
 citecolor=blue}

\makeatletter
\numberwithin{equation}{section}
\numberwithin{figure}{section}
\theoremstyle{plain}
\newtheorem*{cor*}{\protect\corollaryname}
\theoremstyle{plain}
\newtheorem{thm}{\protect\theoremname}[section]
\theoremstyle{definition}

\theoremstyle{remark}
\newtheorem{rem}[thm]{\protect\remarkname}
\theoremstyle{plain}

\theoremstyle{plain}
\newtheorem{lem}[thm]{\protect\lemmaname}
\theoremstyle{plain}
\newtheorem{cor}[thm]{\protect\corollaryname}

\usepackage[utf8]{inputenc}
\usepackage[T1]{fontenc}
\usepackage{lmodern}

\usepackage{geometry}
\usepackage{amsmath}

\numberwithin{equation}{section}
\numberwithin{figure}{section}
\usepackage{enumitem}		
 \let\footnote=\endnote
\@ifundefined{lettrine}{\usepackage{lettrine}}{}

\theoremstyle{definition}

\def\R{\mathbb{R}}

\def\N{\mathbb{N}}
\def\Z{\mathbb{Z}}

\def\ep{\varepsilon}

\usepackage{float}

\subjclass[2000]{}

\def\R{\mathbb{R}}

\def\ep{\varepsilon}

\def\M{\mathcal{M}}

\def\gl2{\text{GL}_2(\R)}
\def\sl2{\text{SL}_2(\R)}

\makeatother
\usepackage[T1]{fontenc}
\usepackage[utf8]{inputenc}
\def\ex{\text{Ext}}

\usepackage[toc]{appendix}

  \providecommand{\corollaryname}{Corollary}
  \providecommand{\definitionname}{Definition}
  \providecommand{\lemmaname}{Lemma}
  \providecommand{\propositionname}{Proposition}
  \providecommand{\remarkname}{Remark}
  \providecommand{\theoremname}{Theorem}
\providecommand{\theoremname}{Theorem}
\usepackage{tikz,xcolor,hyperref}

\definecolor{orcidlogocol}{HTML}{A6CE39}
\tikzset{
  orcidlogo/.pic={
    \fill[orcidlogocol] svg{M256,128c0,70.7-57.3,128-128,128C57.3,256,0,198.7,0,128C0,57.3,57.3,0,128,0C198.7,0,256,57.3,256,128z};
    \fill[white] svg{M86.3,186.2H70.9V79.1h15.4v48.4V186.2z}
                 svg{M108.9,79.1h41.6c39.6,0,57,28.3,57,53.6c0,27.5-21.5,53.6-56.8,53.6h-41.8V79.1z M124.3,172.4h24.5c34.9,0,42.9-26.5,42.9-39.7c0-21.5-13.7-39.7-43.7-39.7h-23.7V172.4z}
                 svg{M88.7,56.8c0,5.5-4.5,10.1-10.1,10.1c-5.6,0-10.1-4.6-10.1-10.1c0-5.6,4.5-10.1,10.1-10.1C84.2,46.7,88.7,51.3,88.7,56.8z};
  }
}

\newcommand\orcidicon[1]{\href{https://orcid.org/#1}{\mbox{\scalerel*{
\begin{tikzpicture}[yscale=-1,transform shape]
\pic{orcidlogo};
\end{tikzpicture}
}{|}}}}

\usepackage{hyperref} 

\date{\today}

\begin{document}

\title[Periodic approximation of Lyapunov exponents and Joint Spectral Radius]{Periodic Approximation of Topological Lyapunov Exponents and the Joint Spectral Radius for Cocycles of Mapping Classes of Surfaces}

\author{Anders Karlsson}
\address[Anders Karlsson]
        {Section de mathématiques, Université de Genève, Switzerland; Department of Mathematics, Uppsala University, Box 480, SE-75106, Uppsala, Sweden}
\email{anders.karlsson@unige.ch}

\author{Reza Mohammadpour }
\address[Reza Mohammadpour]
        {Department of Mathematics, Uppsala University, Box 480, SE-75106, Uppsala, Sweden}
\email{reza.mohammadpour@math.uu.se}

\date{\today}

\begin{abstract}
We study cocycles taking values in the mapping class group of closed surfaces and investigate their leading topological Lyapunov exponent. Under a natural closing property, we show that the top topological Lyapunov exponent can be approximated by periodic orbits. We also extend the notion of the joint spectral radius to this setting, interpreting it via the exponential growth of curves under iterated mapping classes. Our approach connects ideas from ergodic theory, Teichmüller geometry, and spectral theory, and suggests a broader framework for similar results.
\end{abstract}

\maketitle

\section{Introduction}
The subject of ergodic cocycles of linear transformations is a very developed and advanced one \cite{  arnold1998random, Avila-Bochi, AV07-acta, duarte2016lyapunov, viana2014lectures}. Oseledec's multiplicative ergodic theorem can be said to be a fundamental theorem in this context establishing the a.e. existence of Lyapunov exponents. This is a foundational result for smooth dynamics since the derivative of the iterates of a diffeomorphism of a compact manifold gives rise to such a cocycle.

Since invertible matrices act by isometry on the associated symmetric space, Oseledec's theorem is a special case of an ergodic theorem for random products in a general metric space setting \cite{KaMa, KLed, GoKa, Gouezel2018}. Another example of the metric setting is ergodic cocycles of surface homeomorphisms, or more precisely mapping class elements, considered in \cite{Karlsson-Thurston}, which used Thurston's asymmetric metric on Teichmuller spaces. Horbez \cite{Horbez2016} extended this for i.i.d. products with a theorem that reads as an analog of Oseledec' theorem and an extension of the spectral theorem of Thurston \cite{Thurston1988}:

\begin{thm}[{\cite{Thurston1988, Karlsson-Thurston, Horbez2016}}]\label{Horbez-thm} Let
$v(n,g)=g_n g_{n-1} \ldots g_1$ be a product of random homeomorphisms where $g_i$
are chosen independently and distributed with a probability measure of finite first moment. Then there is (random) filtration of
subsurfaces $Y_{1}\subset Y_{2}\subset...\subset Y_{k}=\Sigma$ and
(deterministic) topological Lyapunov exponents $\lambda_{1}<\lambda_{2}<..<\lambda_{k}$
such that
\[
\lim_{n\rightarrow\infty}\frac{1}{n}\log l_{h}(v(n,g))\alpha)=\lambda_{i}
\]
whenever the simple closed curve $\alpha$ can be isotoped to a curve contained
in $Y_{i}$ but not in $Y_{i-1}$. Here $l_{h}$ is the minimal
length in the isotopy class in some fixed Riemannian metric $h.$ 
\end{thm}

Such cocycles, also called random walks, on mapping class groups were also studied in \cite{kaimanovich1996poisson, karlsson2006laws, Tiozzo2015}.  These papers prove sublinearity of distance from the random walk to Teichmüller or Weil-Peterson geodesics rays (similar to \cite{KaMa}). Note that Theorem \ref{Horbez-thm} is more easily interpretable, since it does not refer to an auxiliary Teichmüller space.

In view of this, it makes sense to wonder what results for linear cocycles extend to the setting of the mapping class groups (MCGs). In this article, we establish a few such results. 
The results can probably be generalized to other contexts.

One well-known result in the linear cocycles setting is a breakthrough result by Kalinin \cite{kalinin2011livsic}, which states that under some conditions, the Lyapunov exponents of linear cocycles can be approximated by the Lyapunov exponents at periodic points.  In Theorem \ref{main-thm}, we obtain an approximation result in the MCG setting similar to \cite[Theorem 1.4]{MR3904184}.

Let $\textbf{A}$ be a norm-bounded non-empty set of complex matrices $d \times d$.  Rota and Strang \cite{RotaStrang1960} introduced the joint spectral radius of \textbf{A},
 is defined to be the quantity
\[ \varrho(\textbf{A}):=\lim_{n\to \infty} \sup \{\|A_{n}\ldots A_1\|^{1/n}: A_i \in \mathbf{A} \},\]

where $\| \cdot \|$ denotes any norm on $\mathbb{C}^d$. Berger and Wang \cite{BergerWang1992} proved the following formula
\[\varrho(\textbf{A})=\limsup _{n \rightarrow \infty} \sup \left\{\rho\left(A_n \cdots A_1\right)^{1 / n}: A_i \in \mathbf{A}\right\},\]
 where  $\rho(A)$ denotes the ordinary spectral radius of a matrix $A$.

The joint spectral radius naturally emerges in various areas such as control theory and stability theory \cite{Barabanov1988, Gurvits1995} coding theory \cite{Moision2001}, wavelet regularity \cite{DaubechiesLagarias1992, DaubechiesLagarias1992b}, numerical solutions of ordinary differential equations \cite{GuglielmiZennaro2001}, and combinatorics \cite{DumontSidorovThomas1999}. Consequently, the problem of determining the joint spectral radius for a finite set of matrices has attracted significant research attention \cite{BlondelNesterov2005, Gripenberg1996, LagariasWang1995, ParriloJadbabaie2007, TsitsiklisBlondel1997, Wirth2002, Bui2022, Morris-matherset, Morris2010}. Moreover, the joint spectral radius is strongly related to the ergodic optimization of linear cocycles and zero-temperature limits \cite{Bochi2018, BochiRams2016,BouschMairesse2002, Jenkinson19, Sert, Mohammadpour2022, ChazottesHochman2010, Mohammadpour2020,  Mohammadpour2025}.

Recently, Breuillard and Fujiwara \cite{BreuillardFujiwara2021} extended the concept of the joint spectral radius to groups acting by isometries on nonpositively curved spaces, providing geometric versions of Berger–Wang results that are applicable to $\delta$-hyperbolic spaces and symmetric spaces of noncompact type. In Theorem \ref{main-thm2}, we establish a formula for the joint spectral radius of mapping class groups (MCG) and explore its connection to Lyapunov exponents through the framework of ergodic optimization of Lyapunov exponents.

\smallskip

\section{Statement of the main results}\label{sec:statements}

Let $M$ be an oriented closed surface of genus $g \geqslant 2$. Let $\mathcal{S}$ denote the isotopy classes of simple closed curves on $M$ not isotopically trivial. For a Riemannian metric $h$ on $M$ and a closed curve $\beta$, let $l_h(\beta)$ be the infimum of the length of curves isotopic to $\beta$. Let $\mathcal{T}:=\mathcal{T}(M)$ be the Teichmüller space of $M$ and, for $w\in \mathcal{T}$, the corresponding hyperbolic length of $\alpha \in \mathcal{S}$ is denoted by $l_{w}(\alpha)$. 

A point $w \in \mathcal{T}$ is called \textit{$\epsilon$-thick} (or $\epsilon$-thin) if the length of the shortest curve in $w$ is greater than or equal to $\epsilon$ (or less than $\epsilon$).  When we simply say ‘the thick part’, we mean
that it is the $\ep$-thick part for some $\ep$.

We denote by $\operatorname{Homeo}(M)$  the group of all homeomorphisms of the manifold $M$. Also, we denote by the $\operatorname{Homeo}^{+}(M)$  the subgroup of $\operatorname{Homeo}(M)$ consisting only of the orientation-preserving homeomorphisms and by $\operatorname{Homeo}_0(M)$ the subgroup of $\operatorname{Homeo}(M)$ consisting of those homeomorphisms that are isotopic to the identity map on $M$, respectively. The mapping class group MCG $(M)$ is the group of isotopy classes of orientation-preserving homeomorphisms of $M$ :
\begin{equation}\label{MCG}
   \operatorname{MCG}(M)=\operatorname{Homeo}^{+}(M) / \operatorname{Homeo}_0(M), 
\end{equation}
which acts by automorphisms of $\mathcal{T}(M)$. Thus, every product of homeomorphism gives rise to a product of mapping classes and acts on $\mathcal{T}$. Note that $\operatorname{MCG}(M)$ is a countable discrete group.

Let $X$ be a compact metric space and let $T:X \to X$ be a homeomorphism. We denote by $\mathcal{M}(X,T)$ the space of all $T$-invariant Borel probability measures
on $X$ with the weak$^{\ast}$ topology, and we also denote by $\mathcal{E}(X, T)$ the set of ergodic measures. Assume that $g: X \rightarrow \operatorname{MCG}(M)$ is a measurable map and let
$$
Z_n(x):=g(x) g(T (x)) \ldots g\left(T^{n-1} (x)\right),
$$
which is called a cocycle. Let $f_n=Z_n^{-1}$ and the $g_i=g\left(T^{i-1} x\right)$. We denote $f:=f_1$.  A random walk on MCG can be seen as a special case in which the increments $g\left(T^i x\right)$ are independent and identically distributed. This corresponds to when $X$ is a product space of infinite copies of a fixed probability space, and $T$ represents the shift operator. This is called the \textit{full shift}. In this paper, we are interested in cases beyond the full shift.

For $w \in \mathcal{T}$ denote by $l_w(\alpha)$ the minimal length in its isotopy class in the hyperbolic metric $w$. Let us recall Thurston's asymmetric Lipschitz metric \cite{Thurston, Papadopoulos},
$$
L(w, y)=\log \sup _{\alpha \in \mathcal{S}} \frac{l_y(\alpha)}{l_{w}(\alpha)} .
$$

It is easy to see that $L$ verifies the triangle inequality, and it also true that it separates points although this is non-trivial. Therefore, $L$ satisfies all the axioms for a metric except the symmetry, which indeed fails except in very special cases of surfaces with symmetries. The triangle inequality reads
$$
L(w, z) \leqslant L(w, y)+L(y, z) .
$$
 
Let $\mu$ be an invariant measure on $X$. Fix a base point $o \in \mathcal{T}$. We say that $f$ is an integrable cocycle if
$$
\int_{X}\left(L\left(f(x) o, o\right)+L\left(o, f(x) o\right)\right) d \mu(x)<\infty.
$$

Note that we have the following subadditivity property:
$$
\begin{aligned}
L\left(f_{n+m}(x) o, o\right) & \leqslant L\left( f_m\left(T^n (x)\right)f_n(x) o,f_m\left(T^n (x)\right) o\right)+L(f_m((T^n (x)) o, o) \\
& =L\left(f_n(x) o, o\right)+L\left(f_m\left(T^n (x)\right) o, o\right) .
\end{aligned}
$$

For any ergodic measure-preserving transformation $T$ of a probability space $(X, \mu)$, by the subadditive ergodic theorem of Kingman, one knows that, for $\mu$-a.e. $x$, the following limit exists:
\begin{equation}\label{Lyapunov-exponent}
\lambda_L(\mu, f)=\lim _{n \rightarrow \infty}   \frac{1}{n} L\left(f_n(x) o, o\right)=\lim _{n \rightarrow \infty} \frac{1}{n} L\left(o, Z_n(x) o\right).
\end{equation}

Another metric on $\mathcal{T}$ is the  Teichmüller metric given by the following formula of Kerckhoff \cite{Kerckhoff1981}:
$$
d_T(w, y):=\sup _{\alpha \in \mathcal{S}} \frac{1}{2} \log \frac{\ex_{w}(\alpha)}{\ex_y(\alpha)}.
$$
This is an acutal (symmetric) metric. We refer the reader to \cite{Papadopoulos} for more information on Teichmüller theory. Since the metric is invariant under the action of MCG, cocycles have the corresponding subadditive property for the same reason as above. 
Therefore, one can also define the Lyapunov exponent for $d_T\left(f_{n}(x) o, o\right)$, similar to the Thurston metric as follows. Let $\mu$ be an ergodic invariant measure, then for $\mu$-a.e. $x$, the following
limit exists:
$$
\lambda_{T} (\mu, f)=\lim _{n \rightarrow \infty}   \frac{1}{n} d_{T}\left(f_n(x) o, o\right)=\lim _{n \rightarrow \infty} \frac{1}{n} d_{T}\left(o, Z_n(x) o\right).
$$

We denote
\begin{equation}\label{regular points}
    \Lambda:=\{ x \in X : \eqref{Lyapunov-exponent} \text{ holds} \}.
\end{equation}
and correspondingly for the Teichmüller metric.


We say that a homeomorphism $T$ of a metric space $(X, d)$ satisfies the \textit{(Anosov) closing property} if there exist constants $C, \gamma, \delta_0 > 0$ such that for any $x\in X$ and $k \in \mathbb{N}$ with $d\left(x, T^k (x)\right) < \delta_0$, there exists a point $p \in X$ with $T^k (p) = p$ such that the orbit segments $x, T (x), \ldots, T^k (x)$ and $p, T (p), \ldots, T^k (p)$ are exponentially close. More precisely,
$$
d\left(T^i (x), T^i (p)\right) \leq C d\left(x, T^k (x)\right) e^{-\gamma \min \{i, k-i\}} \quad \text { for every } i = 0, \ldots, k .
$$

Notice that shifts of finite type (for example i.i.d. random products), basic pieces of Axiom
Diffeomorphisms and more generally, hyperbolic homeomorphisms are particular
examples of maps satisfying the closing property (see \cite{Katok1995}).

Our first result concerns the approximation of Lyapunov exponents by periodic points.

\begin{thm}\label{main-thm}Let $X$ be a compact metric space and let $T$ be a homeomorphism of $X$ satisfying the closing property. Let $\mu$ be an $T$-invariant ergodic Borel probability measure on $X$. Assume that $g$ is a locally constant MCG-cocycle over $T$ with respect to the metric $D$ either being the Thurston or Teichmüller metric.   Then for each $\epsilon>0$ there exists a periodic point $p=T^k(p)$ in $X$ such that
$$
\left|\lambda_D(\mu, f)-\frac{1}{k}D(f_k(p)o, o)\right| < \epsilon.
$$
\end{thm}

$$
$$

We will consider the following analog of the joint spectral radius. Let $T:X \to X$ be a homeomorphism of a compact metric space $X$. Assume that $g$ is a locally constant MCG-cocycle over $T$. Let
$$
\rho = \sup_{x\in X} \limsup_{n\rightarrow\infty} \sup_{\alpha \in \mathcal S} \left( \frac{l_h(Z_n(x)\alpha)}{l_h(\alpha)} \right) ^{1/n}.
$$
Since the surface is compact this is independent of the choice of Riemmanian metric $h$. 

The following is now true:

\begin{thm}\label{main-thm2}

Let $T:X \to X$ be a homeomorphism of a compact metric space $X$. Assume that $g$ is a locally constant MCG-cocycle over $T$, and denote by $\rho$  the associated joint spectral radius.  
Then, there is an ergodic measure $\mu$ such that for $\mu$-almost every $\omega$, there is  $\eta \in \mathcal{S}$ such that
$$
\lim_{n \to \infty}l_{h}(Z_n(\omega)\eta)^{\frac{1}{n}}=\rho.
$$
\end{thm}


\begin{rem} Note that since $\eta$ is fixed and we are only considering the exponential growth rate, there is no need to divide by $l_h(\eta)$ in the above statement. One can also replace $l_{h}(f_{n}(x)\eta)^{\frac{1}{n}}$ with $\ex_{h}(f_{n}(x)\eta)^{\frac{1}{n}}.$  This has to do with that when we deal with an orbit in MCG, we stay still in the moduli space so we stay in some thick part as we started in. Therefore,
    by \cite[Theorem B]{ChoiRafi2007} and \cite[Theorem E]{Rafi}, Thurston metric $L$ and  Teichmüller metric $d_T$ are the same up to an additive error. In particular the corresponding Lyapunov exponents are the same. 
\end{rem}

Alternatively, we could formulate the theorem in terms of distances in the Teichmüller spaces. Here is one special case of particular interest.
 Let $g: X \to \operatorname{MCG}(M) $ be a locally constant MCG-cocycle over $(X,T)$. Since $ g $ is a locally constant MCG-cocycle of a compact metric space $X$ and $MCG(M)$ is a discrete group, there exists a finite set of maps $ G := \{g_1, g_2, \ldots, g_k\} $ that generates the map $g$. $G$ 
is called the generator of the cocycle $g$.

Define the following \emph{metric joint spectral radius}
$$
\varrho = \sup_{x \in X} \limsup_{n \rightarrow \infty}\frac 1n L(o,Z_n(x)o). 
$$
It is easy to see that $\varrho =\log \rho$. Note that we could also have used Teichmüller metric since they differ by constants in a fixed thick part, and applying mapping class elements we stand still in the moduli space.   

In the special case of the full shift we can write
$$
\varrho =\limsup_{n\rightarrow \infty} \frac 1n \max \{ L(o,g_1g_2...g_n o):g_i \in G \}.
$$

A special case of the theorem can now be formulated:
\begin{cor}\label{cor}
    Let $T: \{1, \ldots, k\}^{\Z} \to \{1, \ldots, k\}^{\Z}$ be a full shift. Assume that $ G= \{g_1, g_2, \ldots, g_k\} $ generates  a  locally constant MCG-cocycle $g: \{1, \ldots, k\}^{\Z}\to \operatorname{MCG}(M)$ with respect to the metric $D$ either being the Thurston or Teichmüller metric. Denote by $\varrho$ the associated metric joint spectral radius. Then, 
 there is an ergodic measure $\mu$ such that for $\mu$-almost every $\omega$,
 $$
 \lim_{n\rightarrow \infty} \frac 1n D(o,Z_n(\omega)o)= \varrho
 $$
 and there is  $\eta \in \mathcal{S}$ depending on $\omega$ such that 
 $$
 \lim_{n\rightarrow \infty} \frac 1n \log l_h(Z_n(\omega)\eta) =\varrho.
 $$
 
\end{cor}

\section{Proofs of main results}\label{proofs}




We will use the following result in the proof of Theorem \ref{main-thm}.


\begin{lem}[{\cite[Lemma 8]{GG}}]\label{GG} Let $T: X \rightarrow X$ be a homomorphism of a compact metric space $X$ preserving an ergodic Borel probability measure $\mu$. Then there exists a set $P$ with $\mu(P)=1$ such that for each $x \in P$ and $\epsilon, \delta>0$ there exists an integer $N=N(x, \epsilon, \delta)$ such that if $n>N$ then there is an integer $k$ with
$$
n(1+\epsilon)<k<n(1+2 \epsilon) \quad \text { and } d\left(x, T^k (x)\right)<\delta.
$$
\end{lem}

We are now in a position to complete the proof of Theorem \ref{main-thm}. Recall that $\Lambda$ is the set of points where the Lyapunov exponents exist (see \ref{regular points}) in either metric, $\mu$ is an ergodic invariant measure, and $\lambda = \lambda_{D}(\mu, f)$ represents the Lyapunov exponent.
\begin{proof}[Proof of Theorem \ref{main-thm}]
Fix $\epsilon >0$ and a point $x$ in $\Lambda \cap P$, where  $P$  is given by Lemma \ref{GG} (this intersection has a full measure). There is an $n_0$ such that for all $n\geq n_0$ we have:
$$
\left|\lambda_D(\mu, f)-\frac{1}{n}D(f_n(x)o, o)\right| < \epsilon.
$$
Take $\delta>0$ such that $f(x)=f(x')$ whenever $d(x,x')<\delta$. This is possible since $X$ is compact and the cocycle is locally constant.

We fix $N=N\left(x, \epsilon, \delta / C\right)$ given by Lemma \ref{GG}, where $C$ is as in the closing property.  We take $n$ greater than $N$ and $n_0$. Then by Lemma \ref{GG}, there exists $k$ such that $n\left(1+\epsilon\right)<k<n\left(1+2 \epsilon\right)$ and $d\left(x, T^k (x)\right)<\delta / C$. By the closing property, there exists a periodic point $p=T^k (p)$ such that
\begin{equation}\label{eq1}
d\left(T^i (x), T^i (p)\right) \leq \delta e^{-\gamma \min \{i, k-i\}} \leq \delta \quad \text { for every } i=0, \ldots, k.
\end{equation}
This means that 
$$
f(T^i(x))=f(T^i(p))
$$
for all $i=0, \ldots, k$. Therefore also $f_k(x)=f_k(p)$ and since $k\geq n_0$ we can conclude that
$$
\left|\lambda_D(\mu, f)-\frac{1}{k}D(f_k(p)o, o)\right| < \epsilon.
$$
\end{proof}

Now, we prove Theorem \ref{main-thm2}.

\begin{proof}[Proof of Theorem \ref{main-thm2}]

In either invariant distance $D$ on Tecihmüller space, note that \newline $D(o,Z_n(x)0)$ is subadditive:
$$
\begin{aligned}
D\left(o, Z_{n+m}(x) o\right) & \leqslant D\left(o, Z_n(x)o\right) +D\left(Z_n(x)o, Z_n(x)Z_m\left(T^n x\right) o\right) \\
& = D\left(o, Z_n(x) o \right)+D\left(o,Z_m\left(T^n (x)\right) o\right).
\end{aligned}
$$
Hence,
by \cite[Theorem A.3]{Morris-matherset} due to Schreiber, Sturman, Stark and Morris, we have $$
\begin{aligned}
\beta(D,g) & :=\lim_{n \to \infty} \frac{1}{n}\sup _{x \in X}  D(o,Z_n(x) o)\\
&=\sup _{\mu \in \mathcal{E}(X,T)} \lim_{n \to \infty} \frac{1}{n} \int D(o,Z_n o) d \mu\\
& =\lim_{n \to \infty} \sup _{\mu \in \mathcal{M}(X,T)} \frac{1}{n} \int D(o,Z_n o)  d \mu\\
&=\sup_{x\in X} \lim_{n \to \infty} \frac{1}{n} D(o,Z_n(x) o). 
\end{aligned}
$$

By the compactness and the upper semi-continuity of Lyapunov exponents, which follows from Kingman's subadditive ergodic theorem (see \cite{viana2014lectures}), there is in view of the above an ergodic measure $\mu$ such that 
\[  \beta(D,g)=\lim_{n\to \infty}\frac{1}{n}L(o,Z_{n}(\omega)o)\]
for $\mu$-a.e. $\omega$. 

We denote $a(n, \omega):=D\left(o, Z_n(\omega) o\right)$ and now follow an argument in \cite{karlsson-Elements}. By \cite[Proposition 4.2]{KaMa}, given a sequence of $\epsilon_i$ tending to 0, and $\mu$-a.e  $\omega$, there is an infinite sequence of $n_i$ and numbers $K_i$ such that

$$
a\left(n_i, \omega\right)-a\left(n_i-k, T^k (\omega)\right) \geq\left(\beta(D,g)-\epsilon_i\right) k
$$

for all $K_i \leq k \leq n_i$.

Furthermore,  one may assume that \[\left(\beta(D,g)-\epsilon_i\right) n_i \leq a\left(n_i, \omega\right) \leq\left(\beta(D,g)+\epsilon_i\right) n_i\quad  \text{ for all }i.\]

By \cite[Theorem E]{Rafi}, there is a finite set of curves $\mathcal{Y}=\mathcal{Y}_{o}$ such that

$$
L\left(o, y\right)=\log \sup_{\alpha \in \mathcal{S}} \frac{l_y(\alpha)}{l_{o}(\alpha)} \asymp \log \max _{\alpha \in \mathcal{Y}} \frac{l_y(\alpha)}{l_{o}(\alpha)}
$$
up to an additive error.

Now, by the pigeonhole principle, refine $n_i$ so that there exists a curve $\eta$ in $\mathcal{Y}$ that realizes the maximum for each $y=Z_{n_i}(\omega) o$, in other words

$$
l_{Z_{n_i}(\omega) o}\left(\eta\right) \asymp \exp \left(n_i(\beta(D,g) \pm \epsilon_i\right).
$$

Given the way $n_i$ was selected, we have
$$
-\log \sup_{\alpha \in \mathcal{S}} \frac{l_{Z_n(\omega) o}(\eta)}{l_{Z_k(\omega) o}(\eta)} \geq-a\left(n_i-k, T^k (\omega)\right) \geq\left(\beta(D,g)-\epsilon_i\right) k-a\left(n_i, \omega\right)
.$$

Then,
$$
\begin{aligned}
\sup_{x \in X}l_{Z_k(x) o}\left(\eta\right)
& \geq l_{Z_k(\omega) o}\left(\eta\right) \\
& \geq l_{Z_{n_i}(\omega) o}\left(\eta\right) e^{-a\left(n_i, \omega\right)} e^{\left(\beta(D,g)-\epsilon_i\right) k}.
\end{aligned}
$$

By taking the limit,
$$\begin{aligned}
       \beta(D,g) &\leq \liminf_{k\to \infty} \log l_{Z_k(\omega) o}\left(\eta\right)^{\frac{1}{k}}\\
       &\leq \liminf_{k\to \infty}\log\sup_{x \in X}  l_{Z_{k}(x)o}(\eta)^{\frac{1}{k}}.\\
\end{aligned}$$

We  also have the upper bound since, for all sufficiently large $n$,
$$\begin{aligned}
   \log  \frac{l_{Z_{n}(\omega)o}(\eta)}{l_{o}(\eta)} &  \leq  \log  \sup_{x \in X}\frac{l_{Z_{n}(x)o}(\eta)}{l_{o}(\eta)}\\
   &  \leq \sup_{x \in X} \log  \frac{l_{Z_{n}(x)o}(\eta)}{l_{o}(\eta)}\\
   & \leq  \sup_{x \in X} L(o,Z_n(x) o)\\
   & \leq (\beta(D,g) +\epsilon)n. 
\end{aligned}$$

By combining the last two inequalities, letting $n$ tend to infinity, we have
\begin{equation}\label{equlity for the JSP}
\lim_{n\to \infty}l_{Z_{n}(\omega)o}(\eta)^{\frac{1}{n}}=\rho. 
\end{equation}

Since $X$ is compact from the point of view of
exponential growth, any Riemannian metric is equivalent, so we can replace $o$ 
with an
arbitrary metric $h$. So, $l_{Z_{n}(x)o}(\eta)=l_{h}(f_n(x)\eta).$
This completes the proof.
\end{proof}





\subsection{Thermodynamic formalism}
    Let $T:X \to X$ be a homeomorphism of a compact metric space $X$. Assume that $g: X \to \operatorname{MCG}(M) $ is a locally constant MCG-cocycle over $T$ with respect to the metric $D$ either being the Thurston or Teichmüller metric. For $q>0$, we denote by $q\Phi_D:=\{q D(f_{n}o, o)\}_{n\in \N}$ a subadditive potential.

In \cite{Cao2008}, Cao, Feng and Huang established a variational principle for the topological pressure of subadditive families of continuous potentials:
\begin{equation}\label{varitional}
P(T,q\Phi_{D})=\sup \bigg\{h_{\mu}(T)+q\lambda_{D}(\mu, f): \mu \in \mathcal{M}(X, T) \bigg\},
\end{equation}
where $h_{\mu}(T)$ is the measure-theoretic entropy (see \cite{przytycki2010conformal}).

Any invariant measure $\mu_q \in \mathcal{M}(X, T) $ that attains the supremum in \eqref{varitional} is called an \textit{equilibrium state} of $q\Phi_{D}$. 

For $q \in \mathbb{R}_{+}$, we denote by $\operatorname{Eq}(q)$ the collection of equilibrium states of $q\Phi_{D}$. Since the Lyapunov exponent is upper semi-continuous, an equilibrium measure always exists if the entropy map $\mu \mapsto h_{\mu}(T)$ is upper semi-continuous. For example, the entropy map is upper semi-continuous if a continuous map $T: X \to X$ of a compact metric space $X$ is expansive, such as the full shift (see \cite{przytycki2010conformal} for more details). If the entropy map is upper semi-continuous, then $\operatorname{Eq}(q)$ is a non-empty compact convex subset of $\M(X, T)$ (see \cite[Theorem 3.3]{FH}).

\begin{cor}
 Let $T: \{1, \ldots, k\}^{\Z} \to \{1, \ldots, k\}^{\Z}$ be a full shift. Assume that $g: \{1, \ldots, k\}^{\Z} \to \operatorname{MCG}(M) $ is a locally constant MCG-cocycle over $T$ with respect to the metric $D$ either being the Thurston or Teichmüller metric. Then, for any $q>0$, any weak* accumulation point $\mu$ of a family of equilibrium states ($\mu_q$) of potentials $q\Phi_{D}$ is an invariant measure such that 
 \[ \lambda_{D}(\mu, f)=\lim_{n\to \infty}\log \sup_{x\in X}l_{h}(f_{n}(x)\eta)^{\frac{1}{n}}\]
 for some $\eta \in \mathcal{S}$ and any Riemannian metric $h$.    

\end{cor}
\begin{proof}
    It follows from the combination \cite[Theorem 1.1]{Mohammadpour2020} and Corollary \ref{cor}.
\end{proof}

\subsection{Potential Extensions}
We believe that the techniques developed in this work can be extended to a broader class of metrics leading to analogous results for other types of transformations. Examples of such settings include
the Kobayashi pseudo-metric and holomorphic maps, Hofer’s metric and symplectomorphisms \cite{Gromov2007} and metrics on the outer space and automorphisms of free groups.

\subsection*{Acknowledgments.}
 A.K is supported in part by the Swiss NSF grants 200020-200400, 200021-212864, and the Swedish Research Council grant 104651320 and R.M is also supported by the Swedish Research Council grant 104651320. 
\bibliographystyle{abbrv}
\bibliography{251902}

\end{document}